\documentclass[
final
]{dmtcs-episciences}

\usepackage[utf8]{inputenc}
\usepackage{subfigure}

\newtheorem{theorem}{Theorem}[section]

\newtheorem{coro}[theorem]{Corollary}

\usepackage{amsmath}
\usepackage{mathrsfs}
\usepackage{amsfonts}

\usepackage[round]{natbib}

\author{Hongliang Lu\affiliationmark{1}\thanks{Supported by the National Natural Science Foundation of China, No. 11471257 and 11871391.}
  \and Qinglin Yu\affiliationmark{2,3}\thanks{Supported by the Discovery Grant from the Natural Sciences and Engineering Research Council of Canada, and the Shanxi Hundred-Talent Program of Shanxi Province. \\Corresponding email:
yu@tru.ca}}
\title[DMTCS]{Binding Number, Toughness and General Matching Extendability in Graphs}
\affiliation{
  Department of Mathematics, Xi'an Jiaotong University, Xi'an, China\\
  School of Science, Xi'an Polytechnic University, Xi'an, China\\
  Department of Mathematics and Statistics, Thompson Rivers University, Kamloops, BC, Canada}
\received{2018-7-31}
\accepted{2018-12-13}
\begin{document}
\publicationdetails{21}{2019}{3}{1}{4723}
\maketitle
\begin{abstract}
A connected graph $G$ with at least $2m + 2n + 2$ vertices which contains a
perfect matching is $E(m, n)$-{\it extendable}, if for any two sets of
disjoint independent edges $M$ and $N$ with $|M| = m$ and $|N|= n$, there is
a perfect matching $F$ in $G$ such that $M\subseteq F$ and $N\cap F=\emptyset$. Similarly, a connected graph with at least $n+2k+2$ vertices is called $(n,k)$-{\it extendable} if  for any vertex set $S$ of size $n$ and any matching $M$ of size $k$ of $G-S$, $G-S-V(M)$ contains a perfect matching.
Let $\varepsilon$ be a small positive constant, $b(G)$ and $t(G)$ be the binding number and toughness of a graph $G$. The two main theorems of this paper are: for every graph $G$ with sufficiently large order, 1) if $b(G)\geq 4/3+\varepsilon$, then $G$ is $E(m,n)$-extendable and also $(n,k)$-extendable; 2)  if $t(G)\geq 1+\varepsilon$ and $G$ has a high connectivity, then $G$ is $E(m,n)$-extendable and also $(n,k)$-extendable. It is worth to point out that the binding number and toughness conditions for the existence of the general matching extension properties are almost same as that for the  existence of perfect matchings.
\end{abstract}

\section{Introduction}
\label{sec:in}

In this paper, we only consider simple connected graphs.
Let $G$ be a graph with vertex set $V(G)$ and edge set $E(G)$.
A \emph{matching} is a set of independent edges and we often refer a matching with $k$ edges as a $k$-{\it matching}. For a matching $M$,
we use $V(M)$ to denote the set of the endvertices of the edges in $M$
and $|M|$ to denote the number of edges in $M$.  A matching is called a \emph{perfect matching} if it covers all vertices of graph $G$.
For $S\subseteq V(G)$, we write $G[S]$ for the subgraph of $G$
induced by $S$ and  $G - S$ for $G[V(G)\backslash S]$. The number of
odd components (i.e., components with odd order) and the number of components of $G$ are denoted
by $c_0(G)$ and $c(G)$, respectively. Let $N_G(S)$ denote the set of
neighbors of a set $S$ in a graph $G$, and $\kappa(G)$  denote the vertex connectivity of graph  $G$.

Let $M$ be a matching of $G$. If there is a matching $M'$ of $G $
such that $M\subseteq M'$, we say that $M$ can be extended to $M'$
or $M'$ is an \emph{extension} of $M$. Suppose that $G$ is a
connected graph with perfect matchings. If each $k$-matching can be
extended to a perfect matching in $G$, then $G$ is called \emph{$k$-extendable}. To avoid triviality, we require that
$|V(G)|\geq  2k+2$ for $k$-extendable graphs. This family of
graphs was introduced and studied first by \cite{Pl80}. A graph $G$ is called
{\it $n$-factor-critical} if after deleting any $n$ vertices the
remaining subgraph of $G$ has a perfect matching, which was
introduced in \cite{Yu} and was a generalization of the notions of the well-known
factor-critical graphs and bicritical graphs (the cases corresponding to $n = 1$
and $2$, respectively). Note that every connected factor-critical graph is 2-edge-connected (see \cite{Yu}).

Let $G$ be a graph and let $n, k$ be nonnegative integers
such that $|V(G)|\geq  n+2k+2$ and $|V(G)|-n \equiv 0 \pmod{2}$. If deleting any $n$ vertices from $G$ the remaining subgraph of $G$ contains a $k$-matching and moreover, each $k$-matching in the subgraph can be
extended to a perfect  matching, then $G$ is called \emph{$(n, k)$-extendable} (\cite{LY01}). This term can be considered as a general framework to unify the concepts of $n$-factor-criticality and
$k$-extendability. In particular, $(n, 0)$-extendable graphs are exactly
$n$-factor-critical graphs and $(0, k)$-extendable graphs are the same
as $k$-extendable graphs. A graph is called \emph{$E(m,n)$-extendable} if deleting edges of any $n$-matching, the resulted graph is $m$-extendable (\cite{PA96}).  $E(m,0)$-extendability is equivalent to $m$-extendability, and $(n,k)$-extendability and $E(m, n)$-extendability are referred as general matching extensions, which are widely studied in graph theory (see \cite{Plu94, Plu96, Plu}).

For a non-complete graph $G$, its \emph{toughness} is defined by
\[
t(G) = \min_{S\subset V(G)} \frac{|S|}{c(G-S)}
\]
where $S$ is taken over all cut-sets of $G$.  The \emph{binding number} $b(G)$ is defined to be the minimum, taken
over all $S\subseteq V (G)$ with $S\neq \emptyset$ and $N_G(S)\neq V (G)$, of the ratios $\frac{|N_G(S)|}{|S|}$.

Toughness and binding number have been effective graphic parameters for studying factors and matching extensions in graphs.
For instances, 1-tough graphs guarantee the existence of perfect matchings (see \cite{Chvatal73}) and graphs with $b(G) \geq \frac{4}{3}$ contain perfect matchings (see \cite{Wood73}).  There are sufficient conditions with respect to $t(G)$ and $b(G)$ in terms of $m, n, k$ to ensure the existences of $k$-extendability, $E(m,n)$-extendability and other matching extensions (see \cite{Chen95,LY98,Plu88,Plu}). Moreover, \cite{RW02} proved a remarkable result that a graph with $b(G)$ slightly greater than $\frac{4}{3}$ ensure $k$-extendability if the order of $G$ is sufficiently large. Recently, \cite{PS2016} extended this result to $E(m,n)$-extendability. In this paper, we continue the study in this direction and prove that the essential bounds of $t(G)$ and $b(G)$ (i.e., $1$ and $\frac{4}{3}$) which guarantee the existence of a perfect matching can also ensure the existence of all general matching extensions mentioned earlier.

\cite{Tutte47} gave a characterization for a graph to have a perfect matching.
\begin{theorem}[\cite{Tutte47}]\label{Tutte}
Let $G$ be a graph with even order. Then $G$ contains a perfect matching  if and only if for any $S\subseteq V(G)$
$$c_0(G-S)\leq |S|.$$
\end{theorem}

The following result is an extension of Tutte's theorem and also a lean version of a comprehensive structure theorem for matchings, due to Gallai (1964) and Edmonds (1965). See \cite{LoPl86} for a detailed statement and discussion of this theorem.

\begin{theorem}[see \cite{LoPl86}]\label{GE645}
Let $G$ be a graph with even order. Then $G$ contains no perfect matchings if and only if there exists a set $S\subset V(G)$ such that
$$fc(G-S)\ge |S|+2,$$
where $fc(G-S)$ denotes the number of factor-critical components of $G-S$.
\end{theorem}

The proofs of the main theorems require the following two results as lemmas.

\begin{theorem}[\cite{LY01}]\label{LY01}
If $G$ is an $(n, k)$-extendable graph and $n\geq 1,k\geq 2$, then $G$ is also $(n + 2,
k-2)$-extendable.
\end{theorem}

\begin{theorem}[\cite{Plu88-2}]\label{P88}
If a graph $G$ is connected and $k$-extendable graph ($k \geq 1$), then $G-e$ is $(k-1)$-extendable for any edge $e$ of $G$.
\end{theorem}


\section{Binding Number and Matching Extendability}
\label{sec:bing}

\cite{Chen95} proved that
a graph $G$ of even order at least $2m + 2$ is $m$-extendable if $b(G) >
\max\{m, (7m + 13)/12\}$. \cite{RW02} proved a stronger result (in most cases). We state their result in a
simpler but slightly weaker form below.

\begin{theorem}[\cite{RW02}]\label{Bind-k-ex}
For any positive real number $\varepsilon$ and nonnegative integer $m$,
there exists an integer $N = N(\varepsilon, m)$ such that every graph $G$ of even order greater
than $N$ and $b(G)>4/3 + \varepsilon$ is $m$-extendable.
\end{theorem}

In this section, we extend the above result using a different proof technique.
\begin{theorem}\label{Thm2-1}
Let $k,g$ be two positive  integers such that $g\geq 3$ and let  $g_0=2\lfloor\frac{g}{2}\rfloor+1$.    For any positive real number $\varepsilon<\frac{1}{g_0}$,  there exists $N=N(\varepsilon, k, g_0)$ such that for every graph $G$ with order at least $N$ and girth $g$, if $b(G)>\frac{g_0+1}{g_0}+\varepsilon$, then $G$ is $k$-extendable.
\end{theorem}

\begin{proof}
Suppose that the result does not hold. Then there exists a graph $G$ with  order at least $N$   and $b(G)> \frac{g_0+1}{g_0}+\varepsilon$ such that $G$ is not $k$-extendable. By the definition of $k$-extendable graphs, there exists a $k$-matching $M$ such that $G-V(M)$ contains no perfect matchings.
From Theorem \ref{GE645}, there exists $S\subset V(G)-V(M)$ such that
\[
fc(G-V(M)-S)= s+q,
\]
where $q\geq 2$ is even by parity and $s:=|S|$. Let $C_1,\ldots, C_{s+q}$ denote these factor-critical components of $G-S-V(M)$ such that $|C_1|\leq \cdots\leq |C_{s+q}|$.
Without loss of generality, we assume $|C_1|=\ldots =|C_l|=1$. Note that $|C_i| \geq 3$ implies $g(C_i)\geq g$ as $C_i$ is 2-edge-connected. Thus we have $|C_i|\geq g_0$ for $l+1\leq i\leq s+q$.
Write $U=\cup_{i=2}^{s+q}V(C_i)$ and $W=V(G)-U-S-V(M)$. Note that $V(C_1)\subseteq W$ and $s+q \geq 2$. So we have $U \neq \emptyset$ and $W\neq \emptyset$.  One may see that $N(U) \cap W = \emptyset$ and $N(W) \cap U = \emptyset$. Hence $N(U) \not= V(G)$ and $N(W) \not= V(G)$. Denote $r=\max\{2,l+1\}$. Thus we have

\begin{align*}
b(G)&\leq \min\{\frac{|N(U)|}{|U|},\frac{|N(W)|}{|W|}\}\\
&\leq\min\{\frac{2k+s+\sum_{i=r}^{s+q} |C_i|}{r-2+\sum_{i=r}^{s+q} |C_i|}, \frac{|G|-\sum_{i=2}^{s+q} |C_i|}{|G|-2k-s-\sum_{i=2}^{s+q} |C_i|}\}\\
&=\min\{f,h\}
\end{align*}
where $f=\frac{2k+s+\sum_{i=r}^{s+q} |C_i|}{r-2+\sum_{i=r}^{s+q} |C_i|}$ and $h=\frac{|G|-\sum_{i=2}^{s+q} |C_i|}{|G|-2k-s-\sum_{i=2}^{s+q} |C_i|}$.

\medskip
\textbf{ Claim 1.~}   $2k+s > r-2$.
\medskip

This claim is implied by the following inequality:
\begin{align*}
1<\frac{g_0+1}{g_0}+\varepsilon < b(G) \leq f = \frac{2k+s+\sum_{i=r}^{s+q} |C_i|}{r-2+\sum_{i=r}^{s+q} |C_i|},
\end{align*}

\medskip
\textbf{ Claim 2.~}   $\sum_{i=r}^{s+q} |C_i|< g_0(2k+s)$.
\medskip

Suppose that $\sum_{i=r}^{s+q} |C_i|\ge g_0(2k+s)$. By Claim 1, we have
\begin{align*}
b(G)\leq f&\leq \frac{2k+s+g_0(2k+s)}{r-2+g_0(2k+s)}\\
&\leq\frac{2k+s+g_0(2k+s)}{g_0(2k+s)}\\
&= \frac{g_0+1}{g_0},
\end{align*}
a contradiction.

\medskip
\textbf{ Claim 3.~}   $s< \max\{2(g_0-1)k,\frac{2k}{g_0\varepsilon}\}$.
\medskip

Suppose that $s\ge \max\{2(g_0-1)k,\frac{2k}{g_0\varepsilon}\}$.
Since $s\geq 2(g_0-1)k$, we infer that
\begin{align}\label{g0-ratio}
\frac{s(g_0+1)+2k}{g_0s}\leq\frac{g_0}{g_0-1}.
\end{align}

If
\begin{align}\label{small-s}
\frac{g_0+1}{g_0}+\varepsilon < \frac{(g_0+1)s+2k}{g_0s},
\end{align}
then
$s<\frac{2k}{g_0\varepsilon}$, a contradiction. So it is enough for us to show  (\ref{small-s}).
Consider $q< r-1$. Then
 we infer that
\begin{align*}
\frac{g_0+1}{g_0}+\varepsilon < f &\leq \frac{2k+s+g_0(s+q-r+1)}{r-2+g_0(s+q-r+1)} \quad\quad\mbox{(by Claim 1 and $\sum_{i=r}^{s+q} |C_i| \ge g_0(s+q-r+1)$})\\
&=\frac{s(g_0+1)+2k+g_0(q-r+1)}{g_0s+g_0(q-r+1)+r-2}\\
&<\frac{s(g_0+1)+2k+g_0(q-r+1)}{g_0s+g_0(q-r+1)+r-1-q} \\
&=\frac{s(g_0+1)+2k-g_0(r-1-q)}{g_0s-(g_0-1)(r-1-q)}\\
&\leq \frac{(g_0+1)s+2k}{g_0s}. \quad\quad\mbox{(by (\ref{g0-ratio}) and $g_0s+g_0(q-r+1)>q-r+1$)}
\end{align*}
Next, we consider $q\ge r-1$, then
\begin{align*}
\frac{g_0+1}{g_0}+\varepsilon < f &\leq \frac{2k+s+g_0(s+q-r+1)}{r-2+g_0(s+q-r+1)} \quad\quad\mbox{(by Claim 1 and $\sum_{i=r}^{s+q} |C_i| \ge g_0(s+q-r+1)$})\\
&\leq \frac{2k+s+g_0(s+q'-r+1)}{r-2+g_0(s+q'-r+1)} \quad\quad\mbox{(for any $q'$ satisfying $q \geq q' \geq r-1$}) \\
&=\frac{s(g_0+1)+2k}{g_0s+r-2}\\
&\leq \frac{(g_0+1)s+2k}{g_0s}.
\end{align*}
This completes the proof of Claim 3.

\medskip
\textbf{ Claim 4.~}    $l< \max\{2g_0k+1,\frac{2k}{g_0\varepsilon}+1\}$.
\medskip

Suppose that $l\ge \max\{2g_0k+1,\frac{2k}{g_0\varepsilon}+1\}$.
From Claim 3, we have
\begin{align}\label{s-upp}
s<\max\{2(g_0-1)k,\frac{2k}{g_0\varepsilon}\}.
\end{align}
From (\ref{s-upp}), we see $l\geq s+1$ and thus
\begin{align*}
\frac{g_0+1}{g_0}+\varepsilon < \ f & = \frac{2k+s+\sum_{i=r}^{s+q} |C_i|}{r-2+\sum_{i=r}^{s+q} |C_i|}\\
&=\frac{2k+s+\sum_{i=r}^{s+q} |C_i|}{l-1+\sum_{i=r}^{s+q} |C_i|}\\
&\leq \frac{2k+s}{l-1} \quad \quad \mbox{(by Claim 1)}\\
&\leq \frac{2k+l-1}{l-1}\\
&\leq \frac{g_0+1}{g_0}, \quad \quad \mbox{(since $l\geq 2g_0k+1$)}
\end{align*}
a contradiction.

 From Claim 2, we have
 \begin{align}\label{sum-Ci}
\sum_{i=r}^{s+q} |C_i|< g_0(2k+s).
\end{align}
 Thus
\begin{align*}
\frac{g_0+1}{g_0}+\varepsilon < h&=\frac{|G|-\sum_{i=2}^{s+q} |C_i|}{|G|-2k-s-\sum_{i=2}^{s+q} |C_i|}\\
&= \frac{|G|-(r-2)-\sum_{i=r}^{s+q} |C_i|}{|G|-2k-s-(r-2)-\sum_{i=r}^{s+q} |C_i|}\\
&\leq \frac{|G|-(r-2)-g_0(2k+s)}{|G|-2k-s-(r-2)-g_0(2k+s)} \quad\mbox{(by (\ref{sum-Ci}))}\\
&\leq\frac{|G|-l-g_0(2k+s)}{|G|-2k-s-l-g_0(2k+s)} \quad\mbox{(since $r=\max\{2, l+1\}\le l+2$)}\\
&= \frac{|G|-2kg_0-g_0s-l}{|G|-2k-2kg_0-(g_0+1)s-l},
\end{align*}
i.e.,
\begin{align}\label{contra}
\frac{g_0+1}{g_1}+\varepsilon < \frac{|G|-2kg_0-g_0s-l}{|G|-2k-2kg_0-(g_0+1)s-l}.
\end{align}
Claims 2 and 3 imply that $s, l$ are bounded, therefore
\[
\lim_{|G|\rightarrow \infty}  \frac{|G|-2kg_0-g_0s-l}{|G|-2k-2kg_0-(g_0+1)s-l}=1.
\]

For a large $N$, (\ref{contra}) leads to a contradiction when $|G|>N$. This completes the proof.
\end{proof}

Clearly, Theorem \ref{Thm2-1} is a generalization of Theorem \ref{Bind-k-ex}. For connected graphs $G$, the girth $g$ of $G$ is at least three. Setting $g_0=3$, we obtain the following results regarding the general matching extensions (i.e., stronger properties).

\begin{coro}\label{bind-n-k-ex}
Let $n, k$ be two positive integers. For any $\varepsilon<1/3$, there exists $N=N(\varepsilon,n,k)$ such that  if $b(G)>\frac{4}{3}+\varepsilon$ and the order of $G$ is at least $N$, then $G$ is $(n,k)$-extendable.
\end{coro}
\begin{proof}
Since $b(G)>\frac{4}{3}+\varepsilon$, by Theorem \ref{Bind-k-ex}, for a sufficiently large $|G|$, $G$ is $(k+2n)$-extendable or $(0, k+2n)$-extendable. By Theorem \ref{LY01}, $G$ is $(n,k)$-extendable.
\end{proof}

With similar discussion as in Corollary \ref{bind-n-k-ex}, we can deduce $E(m, n)$-extendability with the same conditions, which is a result proved in \cite{PS2016} but here we gave a much shorter proof.

\begin{coro}\label{bind-E(m,n)}
Let $m,n$ be two positive integers. For any $\varepsilon<\frac{1}{3}$, there exists $N=N(\varepsilon,m,n)$ such that for every graph $G$ with order at least $N$, if $b(G)>\frac{4}{3}+\varepsilon$, then $G$ is $E(m,n)$-extendable.
\end{coro}
\begin{proof}
Since $b(G)>\frac{4}{3}+\varepsilon$, by Theorem \ref{Bind-k-ex}, for a sufficiently large $|G|$, $G$ is $(m+n)$-extendable. Let $M = \{e_1, e_2, \dots, e_n\}$ be any $n$-matching. By Theorem \ref{P88}, $G_1=G-e_1$ is $(m+n-1)$-extendable. Applying Theorem \ref{P88} recursively, we conclude that $G_n=G-\{e_1, e_2, \dots, e_n\}$ is $m$-extendable, that is, $G$ is $E(m,n)$-extendable.
\end{proof}

\noindent\textbf{Remark:} Clearly, Corollaries \ref{bind-n-k-ex} and \ref{bind-E(m,n)} can be easily stated in terms of the more general condition $b(G)>\frac{g_0+1}{g_0}+\varepsilon$. However, without the parameter $g$, the results look more neatly.

\section{Toughness and Matching Extendability}
\label{sec:toughness}

It is not hard to construct examples with any given large toughness, but do not have $(n, k)$-extendability or $E(m, n)$-extendability. Therefore toughness alone is insufficient to guarantee the general matching extension properties. However, with an additional condition in terms of connectivity, it only requires slightly large than 1-toughness to deduce the desired matching extendability.

\begin{theorem}\label{tough}
Let $n$ be a positive integer, $\varepsilon$ be a small positive constant and $G$ be a graph with $t(G)\geq 1+\varepsilon$ and $|V(G)|  \equiv n \pmod 2$. If $\kappa(G)>\frac{(n-2)(1+\varepsilon)}{\varepsilon}$, then $G$ is $n$-factor-critical.
\end{theorem}

\begin{proof}
Suppose that $G$ is not $n$-factor-critical. By the definition of $n$-factor-critical, there exists a subset $S$ of order $n$ such that $G-S$ contains no perfect matchings. By Theorem \ref{Tutte}, there exists $T\subseteq V(G)-S$ such that
\[
q=c_0(G-S-T)\geq |T|+2.
\]
Note that $q\geq 2$. So
\begin{align*}
1+\varepsilon\leq t(G)&\leq \frac{|S|+|T|}{|T|+2}\\
&\leq\frac{\kappa}{\kappa-n+2}, \quad \quad \mbox{(since $\kappa \leq n+|T|$)}
\end{align*}
which implies
\[
\kappa\leq \frac{(n-2)(1+\varepsilon)}{\varepsilon},
\]
a contradiction. This completes the proof.
\end{proof}

\noindent\textbf{Remark:} The connectivity condition in the theorem is sharp. Let $n,t$ be two positive integers and $\varepsilon$ be a small constant such that  $n+t<\frac{(n-2)(1+\varepsilon)}{\varepsilon}$.  Let $G_1 = K_{n+t}$, $G_2 = (t+1)K_1$, and  $G_3 = K_{r}$ ($r$ is any positive integer). Define $G = G_1+(G_2 \cup G_3)$, that is, $G$ is a graph obtained by connecting each vertex in $G_1$ to each vertex in $G_2$ and $G_3$. Let $S=V(G_1)$.  Then $S$ is a cut set of $G$ and thus $\kappa \leq n+t \leq \frac{(n-2)(1+\varepsilon)}{\varepsilon}$. It is easy to verify that
\[
t(G) = \frac{|S|}{c(G-S)}=\frac{n+t}{t+2}\geq 1+\varepsilon.
\]
However, for any set $R$ of $n$ vertices in $S$, $G-R$ has no perfect matchings. So $G$ is not $n$-factor-critical.\\

From Theorem \ref{tough}, it is easy to see the following.

\begin{coro}
Let $n,k$ be two positive integers.
Let $\varepsilon$ be a positive constant and $G$ be a graph with $t(G)\geq 1+\varepsilon$. If $\kappa(G)>\frac{(2k-2)(1+\varepsilon)}{\varepsilon}$, then $G$ is $k$-extendable.
\end{coro}

With the same arguments as in the proof of Corollary \ref{bind-E(m,n)}, Theorem \ref{tough} implies the following.

\begin{coro}
Let $m,n$ be two positive integers. Let $\varepsilon$ be a positive constant and $G$ be a graph with $t(G)\geq 1+\varepsilon$. If $\kappa(G)>\frac{(2m+2n-2)(1+\varepsilon)}{\varepsilon}$, then $G$ is $E(m,n)$-extendable.
\end{coro}


\acknowledgements
\label{sec:ack}
The authors are grateful to an anonymous referee for his/her useful suggestions.

\nocite{*}
\bibliographystyle{abbrvnat}
\bibliography{references}
\label{sec:biblio}

\end{document}